\documentclass[11pt]{amsart}

\usepackage[margin=1in]{geometry}

\usepackage{amsmath, amscd, amsthm, amssymb,hyperref}

\def\C{{\mathbf C}}
\def\R{{\mathbf R}}
\def\Z{{\mathbf Z}}
\def\Q{{\mathbf Q}}
\def\A{{\mathbf A}}
\def\g{{\mathfrak g}}

\def\O{{\mathbb O}}

\newtheorem{theorem}{Theorem}[section]

\newtheorem{lemma}[theorem]{Lemma}

\newtheorem{corollary}[theorem]{Corollary}

\theoremstyle{definition}

\theoremstyle{remark}
\newtheorem{remark}[theorem]{Remark}

\newcommand{\mm}[4]{\left(\begin{smallmatrix} #1 & #2\\ #3 & #4\end{smallmatrix}\right)}

\DeclareMathOperator{\tr}{tr}

\DeclareMathOperator{\SO}{SO}

\DeclareMathOperator{\Sp}{Sp}

\DeclareMathOperator{\SU}{SU}

\DeclareMathOperator{\SL}{SL}
\DeclareMathOperator{\GL}{GL}
\DeclareMathOperator{\PGL}{PGL}

\begin{document}
\title{Computation of Fourier coefficients of automorphic forms of type $G_2$}
\author{Aaron Pollack}
\address{Department of Mathematics\\ University of California San Diego\\ La Jolla, CA USA}
\email{apollack@ucsd.edu}
\thanks{Funding information: AP has been supported by the NSF via grant numbers 2101888 and 2144021.}

\begin{abstract}In a recent work, we found formulas for the Fourier coefficients of automorphic forms of type $G_2$: holomorphic Siegel modular forms on $\Sp_6$ that are theta lifts from $G_2^c$, and cuspidal quaternionic modular forms on split $G_2$.  We have implemented these formulas in the mathematical software SAGE. In this paper, we explain the formulas of our recent paper and the SAGE implementation.  We also deduce some theoretical consequences of our SAGE computations.\end{abstract}
\maketitle

\setcounter{tocdepth}{1}
\tableofcontents
\section{Introduction} 
The purpose of this paper is to give computations of Fourier coefficients of automorphic forms of type $G_2$.  The automorphic forms we compute come in two flavors.  First, there are vector-valued holomorphic Siegel modular forms on $\Sp_6$, that are exceptional theta lifts from algebraic modular forms for the group $G_2^c$.  Here $G_2^c$ is a group of type $G_2$ that is split at every finite place and for which $G_2^c(\R)$ is compact.  The second sort of automorphic forms we work with are quaternionic modular forms on \emph{split} $G_2$.  These arise as exceptional theta lifts form algebraic modular forms for the group $F_4^c$, which is a group of type $F_4$ that is split at every finite place and for which $F_4^c(\R)$ is compact.

That it is possible to compute exactly the Fourier coefficients of these exceptional theta lifts is a consequence of the results of \cite{pollackETF}.  In an accompanying SAGE \cite{sagemath} file \cite{pollackGitHub}, we have implemented the formulas of \cite{pollackETF} in the case of level one, so that these Fourier coefficients are exactly determinable on a computer.  The resulting computer calculations are not only satisfying--for example, they give a partial simultaneous check on the formulas of \cite{dalal} and \cite{pollackETF}--but they also have theoretical consequences.

We now say a bit more about some theoretical consequences of these calculations, deferring a discussion of the results of \cite{pollackETF} and the SAGE computations to later sections.
\subsection{Siegel modular forms of genus three}
Let $\Sp_{2n}$ denote the symplectic group, consisting of matrices $g$ with $g \mm{0_n}{1_n}{-1_n}{0_n} g^t = \mm{0_n}{1_n}{-1_n}{0_n}$.  Recall that holomorphic Siegel modular forms are a certain type of automorphic forms for this group.  Explicitly, set $S_n$ the $n \times n$ symmetric matrices, and denote by $\mathcal{H}_n = \{Z= X + iY : X, Y \in S_n(\R), Y > 0\}$ the so-called Siegel upper half-space of degree $n$.  Here $Y > 0$ means that $Y$ is positive-definite.  The group $\Sp_{2n}(\R)$ acts on $\mathcal{H}_n$ via the formulas $\mm{a}{b}{c}{d} Z = (aZ+b)(cZ+d)^{-1}$.  If $g = \mm{a}{b}{c}{d} \in \Sp_{2n}(\R)$, and $Z \in \mathcal{H}_n$, set $J(g,Z) = cZ+d$, which lives in $\GL_n(\C)$.   Let $(\rho,V)$ be a finite-dimensional algebraic representation of $\GL_n(\C)$.  If $\Gamma \subseteq \Sp_{2n}(\Z)$ is a congruence subgroup, a holomorphic Siegel modular form of weight $\rho$ and level $\Gamma$ is a holomorphic function $f: \mathcal{H}_n \rightarrow V$ satisfying $f(\gamma Z) = \rho(J(\gamma,Z)) f(Z)$ for all $\gamma \in \Gamma$; one also imposes a moderate growth condition.

If $\Gamma = \Sp_{2n}(\Z)$, we say that $f$ has level one.  Holomorphic Siegel modular forms have a Fourier expansion, which we explicate in the level one case.  Denote $S_n(\Z)^\vee$ the half-integral symmetric $n \times n$ matrices. I.e., $T \in S_n(\Q)$ is in $S_n(\Z)^\vee$ if the diagonal entries of $T$ are integers, and the off-diagonal entries are integers divided by $2$.  If $f$ is a level one Siegel modular form of weight $(\rho,V)$, then one can write $f(Z) = \sum_{T \in S_n(\Z)^\vee: T \geq 0}{ a_f(T) e^{2\pi i (T,Z)}}$.  Here $T \geq 0$ means that $T$ is positive semi-definite, $(T,Z) = \tr(TZ) \in \C$, and $a_f(T) \in V$.  The vectors $a_f(T)$ are called the Fourier coefficients of $f$.  It is known \cite{ibukiyama} that, given $\rho$, there exists an explicitly determinable finite set $C_\rho$ of half-integral symmetric matrices, so that if $f$ is a level one Siegel modular form of weight $\rho$ and $a_f(T) = 0$ for all $T \in C_\rho$, then $f = 0$.  Consequently, computing finitely many Fourier coefficients of a level one Siegel modular form completely determines it.

Recall that the dual group of $\Sp_6$ is $\SO_7(\C)$, which receives a map from $G_2(\C)$, the dual group of $G_2$.  Langlands functoriality thus predicts a lifting of automorphic representations on groups of type $G_2$ to automorphic representations of $\Sp_6$.  One of the results of \cite{pollackETF}, combined with work of Gross-Savin \cite{grossSavin}, Magaard-Savin \cite{magaardSavin}, and Gan-Savin \cite{ganSavinLLC}, allows one a way to computationally and provably produce instances of this lift.  We have implemented this computation in SAGE, and one consequence of the SAGE computations is the following theorem.

To setup the theorem, recall that Chenevier-Taibi \cite{chenevierTaibi} have computed the dimension of the space of level one holomorphic Siegel modular forms of various weights.  The following table is built from their computations.

\[ \begin{array}{|c|c|c|} \hline (k_1,k_2) & \lambda = (k_1+2k_2+4, k_1+k_2+4, k_2+4) & m(\lambda) \\ \hline
(0,4) & (12,8,8) & 1\\ (2,4) & (14,10,8) & 1\\ (3,3) & (13,10,7) & 1\\ (0,6) & (16,10,10) & 2\\ (3,4) & (15,11,8) & 1\\ (6,2) & (14,12,6) & 1\\ (5,3) & (15,12,7) & 1 \\ (7,2) & (15,13,6) & 1 \\ (9,1) & (15,14,5) & 1 \\ (6,3) & (16,13,7) & 2 \\ (8,2) & (16,14,6) & 2 \\ \hline  \end{array} \]

In the table, $m(\lambda)$ denotes the dimension of the space of level one cuspidal Siegel modular forms for $\Sp_6$ of weight $\lambda$.  The parameter $(k_1, k_2)$ has to do with algebraic modular forms on $G_2^c$, and will be explained in section \ref{sec:Siegel}
\begin{theorem}\label{thm:Sp6} If $\lambda$ is in the above table, then every level one cuspidal eigenform of weight $\lambda$ has all Satake parameters in $G_2(\C)$.
\end{theorem}

\subsection{Quaternionic modular forms on split $G_2$}
While the paper \cite{pollackETF} discusses the holomorphic Siegel modular forms of genus three, its main results concern the (quaternionic) modular forms on split $G_2$. Recall that the split group $G_2$ of this Dynkin type does not have a symmetric space $G_2(\R)/K$ with a $G_2(\R)$-invariant complex structure.  Thus, there is no notion of holomorphic modular forms on this group.   

A great replacement for the holomorphic modular forms was found by Gross-Wallach \cite{grossWallachII} and Gan-Gross-Savin \cite{ganGrossSavin}.  To briefly describe these objects, let us begin by recalling that the maximal compact subgroup $K$ of $G_2(\R)$ is $(\SU(2) \times \SU(2))/\mu_2$, where the first $\SU(2)$ is the so-called long root $\SU(2)$, and the second is the short root $\SU(2)$.  Forgetting the second $\SU(2)$ factor, one has a surjection $K \rightarrow \SU(2)/\mu_2 \simeq \SO(3)$.  For a positive integer $\ell$, let $\mathbf{V}_{\ell}$ be the irreducible complex representation of $K$ that is the pull-back of the $2\ell+1$ irreducible dimensional representation of $\SO(3)$.

If $\Gamma \subseteq G_2(\R)$ is a congruence subgroup, a level $\Gamma$ quaternionic modular form on $G_2$ of weight $\ell$ is a smooth function $\varphi: \Gamma \backslash G_2(\R) \rightarrow \mathbf{V}_{\ell}$ of moderate growth satisfying
\begin{enumerate}
\item $\varphi(gk) = k^{-1} \varphi(g)$ for all $k \in K$;
\item $\mathcal{D}_{\ell} \varphi \equiv 0$, for a certain linear, first-order differential operator $\mathcal{D}_{\ell}$.
\end{enumerate}
The modular form $\varphi$ is cuspidal if and only if $\varphi$ is bounded.  We say $\varphi$ is of level one if $\Gamma = G_2(\Z)$.

The quaternionic modular forms of weight $\ell$ have a Fourier expansion, similar in spirit to the Fourier expansion of holomorphic Siegel modular forms.  We explicate this Fourier expansion in the level one case: For every real binary cubic $f(u,v) = au^3 + bu^2v + cuv^2 + dv^3$, there is a completely explicit moderate growth function (defined in terms of $K$-Bessel functions) $W_{f,\ell}: G_2(\R) \rightarrow \mathbf{V}_{\ell}$ satisfying properties (1), (2) of the definition of quaternionic modular forms of weight $\ell$.  The function $W_{f,\ell}(g)$ is $0$ if the discrimant of the cubic $f$ is negative.  And, if $\varphi$ is a level one cuspidal quaternionic modular form of weight $\ell$, then $\varphi_Z(g) = \sum_{f \text{ integral}}{a_\varphi(f) W_f(g)}$.  Here $\varphi_Z(g)$ is a certain compact integral transform of $\varphi$, which uniquely determines it, and the $a_\varphi(f)$ are complex numbers, called the Fourier coefficients of $\varphi$.

While \emph{a priori} the Fourier coefficients of cuspidal quaternionic modular could be transcendental numbers, the main result of \cite{pollackETF} is that if $\ell \geq 6$ is even, then there is a basis of the space of level $\Gamma$, weight $\ell$ cuspidal quaternionic modular forms whose Fourier coefficients all lie in the cyclotomic extension of $\Q$.  Moreover, the proof is constructive, giving the exact Fourier expansion of such $G_2$-cusp forms in an explicitly computable way.  We have implemented these explicit formulas of \cite{pollackETF} in the case of level one forms on $G_2$.  This allows one to compute finitely many Fourier coefficients of many level one cusp forms on $G_2$.

Recall that, associated to every integral binary cubic form $f$ is a cubic ring $S_f$.  If $f_1, f_2$ are in the same $\GL_2(\Z)$ orbit, then $S_{f_1}$ is isomorphic to $S_{f_2}$.  It follows immediately from the proof of this correspondence that, associated to $f$ is in fact also an orientation of $S_f /\Z$, by which we mean a generator of $\wedge^2(S_f/\Z)$.  If $f_1$ and $f_2$ are in the same $\SL_2(\Z)$-orbit, then $S_{f_1}$ is isomorphic to $S_{f_2}$ as oriented cubic rings.  Now, it is an easy consequence of the existence of the Fourier expansion of quaternionic modular forms, that if $f_1 = g \cdot f_2 = \det(g)^{-1}f_2((u,v)g)$ for $g \in \GL_2(\Z)$, then then $a_\varphi(f_1) = \det(g)^{\ell} a_{\varphi}(f_2)$.  Here the integer $\ell$ is the weight of $\varphi$.  Thus, if $S$ is an oriented cubic ring, we can write $a_\varphi(S)$ for the associated Fourier coefficient.

Here is one theoretical consequence of our ability to compute finitely many Fourier coefficients of some cuspidal level one quaternionic modular forms on $G_2$.  Recall that Dalal \cite{dalal} has recently given an explicit formula for the dimension of the space $S_{\ell}(G_2(\Z))$ of level one cuspidal quaternionic modular forms of weight $\ell \geq 3$.  It follows from his formulas that  $S_{9}(G_2(\Z))$ and $S_{11}(G_2(\Z))$ are each one-dimensional.  Let $F_9, F_{11}$ denote the eigenforms spanning these spaces.

\begin{theorem}\label{thm:QMF} The eigenforms $F_9$ and $F_{11}$ can be normalized to have all Fourier coefficients in $\Z$, and moreover, all Fourier coefficients $a_f(S)$ for (oriented) cubic rings $S$ of the form $\Z \times B$ are $0$.\end{theorem}

The latter part of Theorem \ref{thm:QMF} gives some evidence for a conjecture of Gross \cite{liChao}, as we now explain.  For every level one cuspidal holomorphic modular form $f$ on $\PGL_2$ of weight $2k$, the Arthur multiplicity conjecture predicts the existence of a cuspidal lift $F_f$ to $G_2$, which is a level one cuspidal quaternionic modular form of weight $k$.  Suppose now that $E$ is a totally real etale cubic algebra, with maximal order $\mathcal{O}_E$.  Gross has suggested that square Fourier coefficient $a_{F_f}(\mathcal{O}_E)^2$ should be related to the central $L$-value $L(f \otimes V_E, 1/2)$, where $V_E$ is the two-dimensional motive attached to $E$, i.e., where $\zeta_E(s) = \zeta(s) L(V_E,s)$.  Suppose now that $E = \Q \times F$ with $F$ real quadratic.  Then $L(f \otimes V_E,s)$ factors as $L(f,s) L(f \otimes \epsilon_F,s)$ for a quadratic character $\epsilon_F$.  Hence if $L(f,1/2) = 0$, then $L(f \otimes V_E,1/2) = 0$.  But when $k$ is odd, the central value $L(f,1/2)$ is indeed $0$, so Gross' conjecture would predict $a_{F_f}(\Z \times B) = 0$ for such $f$.  Thus Theorem \ref{thm:QMF} gives some evidence for this conjecture.

\begin{remark} Because the application of our computation of Fourier coefficients is to (conjectural) lifts from $\PGL_2$, the reader might be concerned that we are only able to compute Fourier coefficients of ``small" automorphic representations on $G_2$.  However, this is not the case.  The quaternionic modular forms on $G_2$ contain those automorphic forms that sit as the minimal $K$-type in automorphic representations $\pi = \pi_f \otimes \pi_\infty$, with $\pi_\infty$ a quaternionic discrete series.  Thus, most of these eigenforms should sit in generic $L$-packets.\end{remark}

\subsection{Acknowledgements} We thank Gaetan Chenevier for fruitful exchanges related to our computations and Chao Li for explaining to us Gross's conjecture \cite{liChao}.  It is also a pleasure to thank Wee Teck Gan, Dick Gross, and Gordan Savin for inspirational mathematics upon which these computations are built.

\section{Exceptional algebra}
We explain in this section a little bit of exceptional algebra, both ``theoretically" and ``computationally".  The results in later sections depend upon this algebra.

\subsection{Octonions}
We begin by recalling the octonions $\O$ with positive-definite norm form.  These are an $8$-dimensional $\Q$-vector space, with a law of composition $\O \times \O \rightarrow \O$ that is bilinear, but neither commutative nor associative.  There is an element $1$ with $1 \cdot x = x \cdot 1 = x$ for all $x \in \O$.  Moreover, there is a positive definite quadratic form $n_{\O}: \O \rightarrow \Q$ that satisfies $n_{\O}(xy) = n_{\O}(x) n_{\O}(y)$ for all $x, y \in \O$.  Let $(\,,\,)$ denote the symmetric bilinear form induced by $n_{\O}$, so that $(x,y) = n_{\O}(x+y)-n_{\O}(x)-n_{\O}(y)$ for all $x,y \in \O$.  Let $V_7 \subseteq \O$ denote the orthogonal complement of $\Q \cdot 1$.  Define an involution, $*$, on $\O$ so that $1^* = 1$ and $x^* = -x$ if $x \in V_7$.  Then $xx^* = x^* x = n_{\O}(x)$ for all $x \in \O$; $(xy)^* = y^* x^*$; and $x + x^* = \tr_{\O}(x) 1$ for an element $\tr_{\O}(x) \in \Q$.  For $z_1, z_2, z_3 \in \O$, define $(z_1, z_2, z_3) = \tr_{\O}(z_1 (z_2 z_3))$.  It turns out that this quantity is equal to $\tr_{\O}((z_1 z_2) z_3)$, and that one has $(z_1, z_2, z_3) = (z_2, z_3, z_1) = (z_3, z_1, z_2)$.

One way to create $\O$ is using quaternion algebras, as follows.  Let $H$ be an arbitrary quaternion $\Q$-algebra, which is ramified at infinity.  Let $\gamma \in \Q^\times$ be negative.  Define $\O = H \oplus H$, with multiplication $(x_1, y_1) (x_2, y_2) = (x_1 x_2 + \gamma y_2^* y_1, y_2 x_1 + y_1 x_2^*)$.  The norm is $n_{\O}((x,y)) = n_H(x) - \gamma n_{H}(y)$, the trace is $\tr_{\O}((x,y)) = \tr_{H}(x)$, and the involution is $(x,y)^* = (x^*, -y)$.  Varying $H$ and $\gamma$, it turns out, gives isomorphic data $(\O,n_{\O}, 1)$, which is why we have dropped $H$ and $\gamma$ from the notation of $\O$.  From now on, we take $H$ to be Hamilton's quaternions, i.e., the unique quaternion $\Q$-algebra ramified at $2$ and $\infty$, with basis $\{1,i,j,k\}$ satisfying $i^2 = j^2 = k^2 = -1$ and $ij = k$.  We take $\gamma = -1$.

There is a maximal $\Z$-order in $\O$, called Coxter's ring of integral octonions \cite{coxeter}.  We denote this ring by $R$.  To construct it, set $e = (0,1)$ and $h =
\frac{1}{2}(i+j+k+e)$. Then, the following are a $\Z$ basis of $R$: $jh, e, -h, j, ih, 1, eh, ke.$ These are the simple roots of the $E_8$ root lattice, with $jh$ the
extended node, $1$ the branch vertex, and $e, -h, j, ih, 1, eh, ke$ going along longways.

The group $G_2^c$ is defined as the algebraic $\Q$-group of linear automorphisms of $\O$ that fixes $1$ and respects the multiplication.  Its Lie algebra can be identified with the kernel of the map $\wedge^2 V_7 \rightarrow V_7$ given by $x \wedge y \mapsto xy-yx$; see \cite{pollackG2} for an explicit basis. An element of $\wedge^2 V_7$ acts on $V_7$ via the formula $x \wedge y (z) = (y,z) x - (x,z) y$; this gives the Lie algebra action of $\g_2 = Lie(G_2^c)$ on $V_7$.  

Base-changing the group $G_2^c$ to $K$, it becomes split.  It is helpful to have a different basis of $\O \otimes K$, in which it is easy to write down nilpotent elements of $\g_2 \otimes K$.  We define
\begin{itemize}
	\item $e_2 = \frac{1}{2}((0,1)-\sqrt{-1}(0,i))$
	\item $e_3^* = \frac{1}{2}((0,j)-\sqrt{-1}(0,k))$
	\item $e_3 = \frac{1}{2}((0,-j)-\sqrt{-1}(0,k))$.
	\item $e_2^* = \frac{1}{2}((0,-1)-\sqrt{-1}(0,i))$
	\item $\epsilon_1 = \frac{1}{2}((1,0)-\sqrt{-1}(i,0))$
	\item $\epsilon_2 = \frac{1}{2}((1,0)+\sqrt{-1}(i,0))$
	\item $e_1 = \frac{1}{2}((j,0)-\sqrt{-1}(k,0))$
	\item $e_1^* = \frac{1}{2}((-j,0)-\sqrt{-1}(k,0))$
\end{itemize}
This basis realizes $\O \otimes K$ as a split quadratic space: All the above basis elements are isotropic, and one has $(\epsilon_1, \epsilon_2) = 1$, $(e_i, e_j^*) = - \delta_{ij}$, and $(\epsilon_i, e_j) = (\epsilon_i, e_j^*) = 0$ for all $i,j$.

One can define elements of $G_2^c(K)$ by exponentiating nilpotent Lie algebra elements.  For SAGE computation below, we will use the following nilpotent elements.
\begin{lemma} \label{lem:nilg2} The elements $u = e_3^*$ and $v= e_1$ form a null pair.  Let $B$ be the Borel of $G_{2,K}^c$ stabilizing the flag $K e_3^* \subseteq K e_3^* + K e_1$.  Then the nilpotent $\overline{\mathfrak{n}}$ in $\g_2 \otimes K$ opposite to $B$ is spanned by the following six elements: 
\begin{enumerate}
\item $\ell_1 = e_1^* \wedge e_2$,
\item $\ell_2 = (\epsilon_1 - \epsilon_2) \wedge e_2^* + e_3 \wedge e_1$
\item $\ell_3 = [\ell_1,\ell_2]$
\item $\ell_4 = [\ell_3,\ell_2]/2$
\item $\ell_5 = [\ell_4,\ell_2]/3$
\item $\ell_6 = [\ell_5,\ell_1]$.
\end{enumerate}
\end{lemma}
\begin{proof} This follows directly from \cite[Section 2.2]{pollackG2}.\end{proof}

\subsection{The exceptional cubic norm structure}\label{subsec:excAlgJ}
We now recall a bigger exceptional algebraic structure.  Namely, let $J = H_3(\O)$ be the $27$-dimensional $\Q$-vector space consisting of elements $X = \left(\begin{array}{ccc} c_1 & x_3 & x_2^* \\ x_3^* & c_2 & x_1 \\ x_2 & x_1^* & c_3 \end{array}\right)$ with $c_1, c_2, c_3 \in \Q$ and $x_1, x_2, x_3 \in \O$.  This $J$ is called the exceptional cubic norm structure.  The space $J$ comes equipped with a cubic norm $N_J: J \rightarrow \Q$ defined as $N_J(X) = c_1 c_2 c_3 - c_1 n_{\O}(x_1) - c_2 n_{\O}(x_2) - c_3 n_{\O}(x_3) + (x_1, x_2, x_3)$.

Let $(\,,\,,\,)_J$ be the unique symmetric trilinear form on $J$ satisfying $(X,X,X)_J = 6 N_J(X)$ for all $X \in J$.  For $X \in J$, let $X^{\#} \in J^\vee$ be the linear map given by $(Z,X^\#) = \frac{1}{2}(Z,X,X)$ for all $Z \in J$.  One says that $X$ has rank one if $X \neq 0$ but $X^\#  = 0$.  For $X,Y \in J$, one sets $X \times Y = (X+Y)^\# - X^\# - Y^\#$, then one has $(Z, X \times Y) = (Z,X,Y)$ for all $Z \in J$.  If $U \in J$ satisfies $N_J(U) = 1$, one defines a symmetric bilinear form on $J$ as $(X,Y)_U = (X,U^\#) (Y,U^\#) - (U, X, Y)$.

We will write elements $X$ of $J$ as $X = [c_1, c_2, c_3; x_1, x_2, x_3]$.  Set $I = [1,1,1;0,0,0]$, which has $N_J(I) = 1$.  Then if $X = [c_1, c_2, c_3; x_1, x_2, x_3]$ and $X' = [c_1', c_2', c_3'; x_1', x_2', x_3']$, then $(X,X')_I = c_1 c_1' + c_2 c_2' + c_3 c_3' + (x_1, x_1')_{\O} + (x_2, x_2')_{\O} + (x_3, x_3')_{\O}$.  Observe that this bilinear form is positive-definite.  It induces an identification $J \simeq J^\vee$, which we denote by $\iota$.  If $X, X'$ are as above, then

\begin{align*} \iota(X^\#) &=[c_2 c_3 - n(x_1), c_3 c_1 - n(x_2), c_1 c_2 - n(x_3); (x_2 x_3)^* - c_1 x_1, (x_3 x_1)^* - c_2 x_2, (x_1x_2)^* - c_3 x_3]\\
\iota(X \times X') &= [c_2 c_3'+c_2'c_3 - (x_1,x_1'), c_3 c_1'+c_3'c_1 - (x_2,x_2'), c_1' c_2+c_1c_2' - (x_3,x_3'); \\ &\,\,\, (x_2' x_3 +x_2 x_3')^*- c_1' x_1-c_1x_1', (x_3' x_1+x_3x_1')^* - c_2' x_2-c_2 x_2', (x_1' x_2+x_1x_2')^* - c_3' x_3-c_3x_3'].\end{align*}

An integral lattice in $J$ is the set $J_R$, consisting of those $X = [c_1, c_2, c_3; x_1, x_2, x_3]$ with $c_i \in \Z$ and $x_i \in R$.  Following \cite{elkiesGrossIMRN}, we distinguish two different quadratic forms on this lattice.  The first is $(\,,\,)_I$.  For the second, define $\beta = \frac{1}{2}(-1+i+j+k,1+i+j+k)$.  Then $\beta^2+\beta+2 = 0$, so that $\tr_{\O}(\beta) = -1$ and $n_{\O}(\beta) = 2$.  One sets $E = [2,2,2;\beta,\beta,\beta]$, so that $N_J(E) = 1$.  The second quadratic form on $J_R$ is $(\,,\,)_E$.  This form is again positive-definite \cite{elkiesGrossIMRN}.

Set $M_J^1$ to be the algebraic $\Q$-group of linear automorphisms of $J$ preserving the cubic norm.  It is a simply-connected group of type $E_6$.  There exists $\delta \in M_J^1(\Q)$ for which $\delta E = I$.  For SAGE computations, we will use an explicit choice of $\delta$.  To set up the result, if $\gamma \in J^\vee$ and $X \in J$, let $\Phi_{\gamma, x}': J \rightarrow J$ be defined as $\Phi_{\gamma,x}(z) = - \gamma \times (x \times z) + (\gamma,z) x + \frac{1}{3}(\gamma,x) z$.  Then $\Phi_{\gamma,x} \in Lie(M_J^1)$ \cite[Proposition 1.1]{PWZ}, \cite[Section 3.3]{pollackQDS}.

For $x, y, z \in \O$, set $V(x,y,z) = [0,0,0;x,y,z]$ and $V_1(x) = V(x,0,0)$, $V_2(y) = V(0,y,0)$ and $V_3(z) = V(0,0,z)$.
\begin{lemma} For $\delta \in M_J^1(\Q)$, one can take
\begin{align*} \delta &= \exp(\Phi'_{e_{22},V_1(-1)}) \exp(\Phi'_{V_1(3/2),e_{22}}) \exp(\Phi'_{e_{22},V_1(-1)}) \exp(\Phi'_{V_1(1),e_{22}})\\ &\times \exp(\Phi'_{e_{22},V_{1}(-\beta/2)}) \exp(\Phi'_{e_{11} V_2(-(\beta+1)/2)}) \exp(\Phi'_{V_3(-\beta/2),e_{11}}).\end{align*}
The $\Phi'_{\gamma,x}$ appearing in this product satisfy $(\Phi'_{\gamma,x})^3= 0$.
\end{lemma}
\begin{proof} The fact that the $(\Phi'_{\gamma,x})^3= 0$ above follows from \cite[Proposition 1.1]{PWZ}.  That the above $\delta$ satisfies $\delta E = I$ can be verified by explicit computation in SAGE.\end{proof}

 Let $F_4^c$ be the algebraic subgroup of $M_J^1$ of elements that also fix $I$.  We define $J^0$ to be the subspace of elements $X \in J$ with $c_1 + c_2 + c_3 = 0$; equivalently, $J^0$ is the orthogonal complement to $I$ under the bilinear pairing $(\,,\,)_I$.  It is preserved by $F_4^c$.  There is a surjective $F_4^c$-equivariant map $\wedge^2 J^0 \rightarrow \mathfrak{f}_4$ from $\wedge^2 J^0$ to the adjoint representation of $F_4^c$. It is given by 
 \[X \wedge Y \mapsto \Phi_{X \wedge Y} :=  \Phi'_{\iota(X),Y} - \Phi'_{\iota(Y),X}.\]
 Let $V_1$ be the kernel of this map.  It is an irreducible representation of $F_4^c$ of dimension $273$.

We will require special vectors $X_{!},Y_{!}$ in $J\otimes K$.  To define them, let $t$ be the square-root of $-1$ in $K$, to distinguish it from $i \in H$ and $i_K \in H_K$.  Then, in $\O_K$, we define the elements

\begin{itemize}
\item $r_1=\frac{1}{2}(0,1-t \cdot i_K)$
\item $r_2=r_1$
\item $r_3 = -t  (i_K, 0)$
\item $s_1= -t  (0,i_K)$
\item $s_2 = \frac{1}{2} (0,1+t \cdot i_K)$
\item $s_3 = -\frac{1}{2} (1+t \cdot i_K,0)$.
\end{itemize}
Now, we set $X_{!} = [1,-1,0,r_1,r_2,r_3]$ and $Y_{!} = [0,-1,1,s_1, s_2, s_3]$ as elements of $J\otimes K$.

\begin{lemma} The wedge $X_{!} \wedge Y_{!} \in \wedge^2 J^0_K$ sits in $V_1$, and is a highest weight vector in $V_1$ for a Borel subgroup of $F_{4,K}^c$.
\end{lemma}
\begin{proof} This is proved in \cite[Lemma 3.1.3 and Example 3.1.4]{pollackETF}.\end{proof}

We will also need other vectors in $J_K$, obtained from $X_{!}$ and $Y_{!}$ by $F_4^c(K)$-automorphisms.  To produce many such automorphisms, we use the following lemma.  Let $e_{11} = [1,0,0;0,0,0]$ and similarly define $e_{22} = [0,1,0;0,0,0], e_{33} = [0,0,1;0,0,0]$.
\begin{lemma} \label{lem:F4nilp} Suppose $v' \in \mathbb{O}\otimes K$ is in the set $\{\epsilon_1, \epsilon_2, e_1, e_2, e_3, e_1^*,e_2^*,e_3^*\}$.  Then  $\Phi'_{u \wedge v}$, with $u = e_{11}, v= V(0,v',0)$, $u = e_{11}, v = V(0,0,v')$, and $u = e_{22}, v = V(v',0,0)$, are nilpotent elements in the Lie algebra $\mathfrak{f}_4$.  They satisfy $(\Phi'_{u \wedge v})^3 = 0$.
\end{lemma}
\begin{proof} For some context regarding this Lie algebra elements, see \cite[page 23]{pollackNotes}.  One can verify in SAGE that the $(\Phi'_{u,v})^3 = 0$.\end{proof}

\subsection{SAGE implementation of octonions}
SAGE already has quaternion algebras implemented.  Using the Cayley-Dickson construction, i.e., realizing $\O$ as $H \oplus H$ with $H$ equal to Hamilton's quaterions, one can realize the octonions in SAGE.  Namely, octonions in SAGE are represented as a vector consisting of two elements of the quaternion algebra $H$, or two elements of the quaternion algebra $H \otimes K$.  Our file, \texttt{g2\_motives.sage}, contains functions to multiply two octonions, take their trace and conjugate, and compute the inner product of two octonions.  Also included is the trilinear form $(x_1, x_2, x_3)_{\O}$.

To do computations, we have some specific bases of $\O$ and of $\O \otimes K$ hard-coded into the program.  The first basis is the ordered Coxeter basis, which is a $\Z$-basis of Coxeter's ring $R$.  The second basis is the split basis of $\O \otimes K$, ordered as $[\epsilon_1, e_1, e_2, e_3, e_1^*, e_2^*, e_3^*, \epsilon_2]$.  There is some built-in code to change basis from one to another, and to go from octonions, to vectors of length $8$ of elements of $\Q$ or $K$.  The Gram matrices for the trace pairing and the norm pairing, with respect to the Coxeter basis, are hard-coded into the file.

\subsection{SAGE implementation of the exceptional cubic norm structure}
Building upon the octonions above, we represent elements of $J$ or $J \otimes K$ as a list of length $6$, $X = [c_1,c_2,c_3, x_1, x_2, x_3]$ where the $c_i \in K$ and the $x_i \in \O \otimes K$ are octonions.  The code contains functions to compute the norm of an element $X$, the element $\iota(X^\#)$ or $\iota(X \times X')$, the pairings $(X,X')_I$ and $(X,X')_E$, and the trilinear form $(X_1, X_2, X_3)_J$.  The Gram matrix for the pairing $(X,X)_E$, for $X \in J_R$, is hard-coded into the file.

\section{Siegel modular forms of genus three}\label{sec:Siegel}
We begin with a discussion of results of \cite{pollackETF} that apply to Siegel modular forms of genus three.  We then discuss the SAGE implementation of the Fourier coefficient formula.

\subsection{Fourier expansion of the holomorphic exceptional theta lift}
In order to put Theorem \ref{thm:Sp6} into context and to explain how it is proved, recall that this result asserts the existence of Siegel modular forms of genus three with all Satake parameters in $G_2(\C)$.  The Siegel modular forms arise as an exceptional theta lift from algebraic modular forms for the group $G_2^c$, as studied in Gross-Savin \cite{grossSavin}.  Specifically, there is a group $H$ of type $E_{7,3}$, and a very special automorphic representation $\Pi_{min,H}$ on $H(\A)$ called the minimal automorphic representation, whose study was begun by Kim \cite{kimMin}.  Gross-Savin use automorphic functions in $\Pi_{min,H}$ to lift automorphic forms on $G_2^c$ to automorphic forms on $\Sp_6$, using the fact that $\Sp_6 \times G_2^c \rightarrow H$.   Namely, if $\alpha$ is an automorphic form on $G_2^c(\A)$ and $\Theta_f \in \Pi_{min,H}$, then one can define the lift 
\[\Theta_f(\alpha)(g) = \int_{G_2^c(\Q) \backslash G_2^c(\A)}{\Theta_f(gh) \alpha(h)\,dh}\]
which is an automorphic form on $\Sp_6$.  These lifts can be made to be vector-valued Siegel modular forms, as proved in \cite{grossSavin}.

Gross-Savin \cite{grossSavin}, Magaard-Savin \cite{magaardSavin} and Gan-Savin \cite{ganSavinLLC} have produced numerous and deep results on this exceptional theta lift.  One question left open by these works, however, is to determine when $\Theta_f(\alpha)$ is nonzero for explicit $\alpha$.  One of the first main results of \cite{pollackETF} solves this question, in the case of level one: It gives explicit formulas to determine the Fourier coefficients of the Siegel modular forms corresponding to the $\Theta_f(\alpha)$.

We explicate some portion of this result, referring to \cite{pollackETF} for more details.  Recall that we denote by $\O$ the octonion algebra over $\Q$ with positive-definite norm form and $R \subseteq \O$ is Coxeter's order of integral octonions.   We write $J = H_3(\O)$ the exceptional cubic norm structure, which is the $27$-dimensional $\Q$-vector space consisting of $3 \times 3$ Hermitian matrices with ``coefficients" in $\O$.  and let $J_R$ be the integral lattice in $J$ consisting of elements whose diagonal entries are in $\Z$ and off-diagonal entries are in $R$.  Recall we denote by $V_7$ the elements of $\O$ with $0$ trace and $K = \Q(\sqrt{-1})$.  For $T \in J$ and $Z \in S_3(\C)$ set $(T,Z) = \frac{1}{2} \tr(TZ + ZT) \in \C$.  Recall that elements of $J$ have a notion of \emph{rank}, which is an integer in $\{0,1,2,3\}$.

Let $W_3$ denote the standard representation of $\GL_3$, with basis $w_1, w_2, w_3$, and let $V_3 = \wedge^2 W_3$ the exterior square representation, with basis $v_1 = w_2 \wedge w_3, v_2 = w_3 \wedge w_1, v_3 = w_1 \wedge w_2$.  Recall that every irreducible algebraic representation  $V$ of $\GL_3(\C)$ sits as the highest weight submodule in $Sym^{k_1}(V_3) \otimes Sym^{k_2}(W_3) \otimes \det(W_3)^{k_3}$ for integers $k_1, k_2, k_3$ with $k_1, k_2 \geq 0$.  The highest weight of such a $V$ is $(k_1+k_2+k_3, k_1 + k_3, k_3)$.  Thus, if $f$ is a level one Siegel modular form of weight $(\rho,V)$, then the Fourier coefficients of $f$ are naturally polynomials in $v_1, v_2, v_3, w_1, w_2, w_3$ of bi-degree $(k_1,k_2)$.

Suppose $u, v \in V_7 \otimes K$ satisfy $u^2 = uv = vu = v^2 = 0$; such a pair is said to be \emph{null}.  Let $T \in J$ have off-diagonal entries $x_1, x_2, x_3 \in \Theta$.  For non-negative integers $k_1, k_2$, set 
\begin{align*} P_{k_1, k_2}(T;u,v) &=  ((x_1,u) v_1 + (x_2,u) v_2 + (x_3,u) v_3)^{k_1} \\ &\times ((x_1 \wedge x_2, u \wedge v)w_3 + (x_2 \wedge x_3, u\wedge v) w_1 + (x_3 \wedge x_1,u \wedge v)w_2)^{k_2}.\end{align*}
Here $(x,u) = \tr(x^* u)$ is the bilinear form associated to the norm on $\Theta$, and $(x \wedge y, u \wedge v) = (x,u)(y,v) - (x,v)(y,u)$.  Moreover, in case either $k_1$ or $k_2$ is equal to $0$, $r^{k}$ is defined to be $1$ if $k=0$, regardless of if $r = 0$.  Fortunately, this is the convention that appears to be used by SAGE.

Finally, so long as at least one of $k_1, k_2$ is positive, set
\[F_{k_1, k_2}^{u,v}(Z) = \sum_{T \in J_R, \mathrm{rank}(T)  = 1}{ \sigma_3(d_T) P_{k_1, k_2}(T;u,v) e^{2\pi i (T,Z)}}\]
where $d_T$ is the largest integer for which $d_T^{-1} T \in J_R$.
\begin{theorem}[See \cite{pollackETF}] \label{thm:sp6} Let the notation be as above, with at least one of $k_1, k_2$ positive.  Let $V \subseteq Sym^{k_1}(V_3) \otimes Sym^{k_2}(W_3) \otimes \det(W_3)^{k_2+4}$ be the highest weight submodule so that the highest weight of $V$ is $(k_1+2k_2+4, k_1+k_2+4, k_2+4)$.  Then $F_{k_1,k_2}^{u,v}(Z)$ is a level one holomorphic Siegel modular form on $\Sp_6$ of weight $(\rho,V)$.
\end{theorem}

One has the following corollary, which is explained in \cite{pollackETF}.
\begin{corollary} \label{cor:sp6} Let the notation be as above, and let $S^{\Theta}_{k_1,k_2}$ be the span of the $F_{k_1,k_2}^{u,v}$ as $u,v$ vary over null pairs. \begin{enumerate}
\item (Gross-Savin) If $k_2 > 0$, the space $S^{\Theta}_{k_1,k_2}$ is contained in the space of cusp forms.
\item The cusp forms in $S^{\Theta}_{k_1,k_2}$ are exactly the level one theta lifts of Gross-Savin, which come from algebraic modular forms on $G_2^c$ of weights $k_1\omega_1 + k_2 \omega_2$, where $\omega_1$ is the highest weight of the $7$-dimensional irrep of $G_2$ and $\omega_2$ is the highest weight of the adjoint representation of $G_2$.
\item (Gross-Savin, Magaard-Savin, Gan-Savin) Suppose $F \in S^{\Theta}_{k_1,k_2}$ is a cuspidal Hecke eigenform.  Then all Satake parameters of $F$ lie in $G_2(\C) \subseteq \SO_7(\C)$.
\end{enumerate}
\end{corollary}

Thus, the theorem and the corollary give explicit ways of constructing cuspidal eigenforms on $\Sp_6$, all of whose Satake parameters are in $G_2(\C)$.  The code implemented produces finitely many Fourier coefficients of the $F_{k_1,k_2}^{u,v}$.  Having run this code, we can obtain Theorem \ref{thm:Sp6}.

\begin{proof}[Proof of Theorem \ref{thm:Sp6}] Suppose $(k_1, k_2)$ is a pair in the table above. Chenevier-Renard \cite{chenevierRenard} have computed the dimension $m(k_1,k_2)$ of the space of algebraic modular forms on $G_2^c$ of weight $(k_1,k_2)$.  From their comutations combined with \cite{chenevierTaibi}, one sees that for the $(k_1,k_2)$ in the table, $m(k_1,k_2) = m(\lambda(k_1,k_2))$. Consider now the map from level one algebraic modular forms on $G_2^c$ of weight $k_1 \omega_1 + k_2 \omega_2$ to polynomials in $v_1, v_2, v_3, w_1, w_2, w_3$ that is the composition of the theta lift with the $T^{th}$ Fourier coefficient, for $T = \frac{1}{2}\left(\begin{array}{ccc} 2 & 1 &1\\ 1 & 2& 1\\ 1 & 1& 2 \end{array}\right)$.  Computing the $T^{th}$ Fourier coefficient of a few $F_{k_1,k_2}^{u,v}$, one finds a space of polynomials of dimension at least $m(k_1,k_2)=m(\lambda(k_1,k_2))$.  Consequently, the theta lift is bijective in these cases.  Applying part (1) and (3) of Corollary \ref{cor:sp6} gives the theorem.\end{proof}

\begin{remark} The theta lift from $G_2^c$ to $\Sp_6$ is not expected to be injective in general.  For example, as explained to the author by Gaetan Chenevier, for every level one cuspidal eigenform $f$ of weight $2k$ for $\PGL_2$, there should be an associated level one algebraic modular form on $G_2^c$ of weight $(k-2) \omega_2$.  And, moreover, when $k$ is odd so that the central $L$-value $L(f,1/2) = 0$, the Arthur multiplicity conjecture predicts that this eigenform should not lift to $\Sp_6$.  As a specific example, when $k = 9$, the dimension $m(0,7) = 1$, but computing a few Fourier coefficients of the $F_{0,7}^{u,v}(Z)$ in SAGE for various specific $u,v$ gives $0$.
\end{remark}

\subsection{Sage implementation to find Fourier coefficients}
If $T = [c_1,c_2, c_3; x_1, x_2, x_3]$ define the projection of $T$ to be $T_0 = [c_1, c_2, c_3; \tr(x_1), \tr(x_2), \tr(x_3)]$, which is a half-integral symmetric matrix.  To calculate the $T_0$ Fourier coefficient of $F^{u,v}_{k_1,k_2}(Z)$ on $\Sp_6$, one must sum $\sigma_3(d_T) P_{k_1,k_2}(T;u,v)$ for all rank one $T \in J_R$ whose projection is $T_0$.  The SAGE code will do this for those $T_0$ with $c_1 = 1$.  Note that in this case, $d_T$ is always equal to $1$, so the factor $\sigma_3(d_T) = 1$.

To find those $T \in J_R^{rk=1}$ with $\mathrm{proj}(T) = T_0$, the code implicitly uses the following easy lemma.
\begin{lemma}\label{lem:Trk1} Suppose $T = [1,b,c;x_1, x_2, x_3] \in J$.  Then $T$ is rank one if and only if $n(x_2) = c$, $n(x_3) = b$ and $x_1 = (x_2 x_3)^*$. 
\end{lemma}

The SAGE code takes in an element $T_0  \in S_3(\Z)^\vee$, with $c_1(T_0) = 1$ and finds all rank one $T$ with projection equal to $T_0$ using Lemma \ref{lem:Trk1}. Then, for each such $T$, the code computes $P_{k_1,k_2}(T;u,v)$ for a given choice of $u,v$, provided by the user, and sums up the results.  The output is a polynomial in the variables $v_1, v_2, v_3, w_1, w_2, w_3$.  To aid the user in constructing null pairs $u,v$, the user must only enter $6$ elements of $K$, $[r_1, r_2, r_3, r_4, r_5, r_6]$.  The code will then produce for you $u_{new} = \exp(n) e_3^*, v_{new} = \exp(n) e_1$, where $n = \sum_{j=1}^{6} r_j \ell_j$ in the notation of Lemma \ref{lem:nilg2}
 
\section{Quaternionic modular forms on $G_2$}
In this section, we give our results on quaternionic modular forms on split $G_2$.  We begin by recalling some of the results of \cite{pollackETF} in this setting, and then explain the proof of Theorem \ref{thm:QMF}.  The proof of this theorem uses SAGE computations, which are also explained in this section.

\subsection{Fourier expansion of the quaternionic exceptional theta lift}
To set up the first result, recall $K = \Q(\sqrt{-1})$.  If $w = (a,b,c,d) \in W_J$, define $pr_I(u,v)$ to be the binary cubic form given as $pr_{I}(u,v) = au^3 + (b,I^\#) u^2 v + (c,I) uv^2 + dv^3$.  Set $J_K = J \otimes K$ and $J_K^0$ to be the trace $0$ elements of $J_K$.  Set $W_J(\Z) = \Z \oplus J_R \oplus J_R^\vee \oplus \Z$.   For $w \in W_{J}(\Z)$, let $d_w$ be the largest integer so that $d_w^{-1} w \in W_{J}(\Z)$. Finally, set $P_{m,I}(w;X,Y) = ((b,X)_I(c,Y) - (b,Y)_I(c,X))^{m}$.

\begin{theorem}[See \cite{pollackETF}] \label{thm:ThetaI} Suppose $X,Y \in J_K^0$ are singular in the sense that $X \wedge Y \in \wedge^2 J_K^0$ is a highest weight vector of $V_1$ for some Borel.  Let $m \geq 1$ be an integer. Then there is a cuspidal quaternionic modular form $\Theta_I(X,Y;m)$ on $G_2$ of level one and of weight $4+m$ with Fourier expansion
\[ \Theta_I(X,Y;m)_Z(g) = \sum_{w = (a,b,c,d) \in W_J(\Z)^{rk = 1}}{ \sigma_4(d_w) P_{m,I}(w;X,Y) W_{pr_I(w),4+m}(g)}.\]
\end{theorem}

For the cubic norm structure $(J_R, E)$, there is an analogous result.
Fix $\delta \in M_J^1(\Q)$ with $\delta E = I$, as in subsection \ref{subsec:excAlgJ}.  Set $(u,v)_E = \frac{1}{4}(E,E,u)(E,E,v)- (E,u,v)$. For $w = (a,b,c,d) \in W_J$, define $pr_E(w) = a u^3 + (b,E^\#) u^2 v + (c,E) uv^2 + d v^3$.  Similar to the above, define $P_{m,E}(w;X_E,Y_E) = ((b,X_E)_E (c, Y_E) - (b,Y_E)_E (c, X_E))^{m}$.
\begin{theorem}[See \cite{pollackETF}] \label{thm:ThetaE} Suppose $X,Y \in J_K^0$ are singular in the sense that $X \wedge Y \in \wedge^2 J_K^0$ is a highest weight vector of $V_1$ for some Borel.  Let $m \geq 1$ be an integer, and set $X_E = \delta^{-1} X$, $Y_E = \delta^{-1} Y$.  Then there is a cuspidal quaternionic modular form $\Theta_{E}(X_E,Y_E;m)$ on $G_2$ of level one and of weight $4+m$ with Fourier expansion
\[\Theta_E(X_E,Y_E;m)_{Z}(g) = \sum_{w = (a,b,c,d) \in W_J(\Z)^{rk=1}}{ \sigma_4(d_w) P_{m,E}(w;X_E,Y_E) W_{pr_E(w),4+m}(g)}.\]
\end{theorem}

We can now explain the proof of Theorem \ref{thm:QMF}, which uses our SAGE implementation to compute finitely many Fourier coefficients of the $\Theta_{E}(X_E,Y_E;m)$.
\begin{proof}[Proof of Theorem \ref{thm:QMF}] Let $m = 5$ or $m=7$.  Computing a single Fourier coefficient of $\Theta_{E}(X_E,Y_E;m)$ for a somewhat randomly chosen $X_E,Y_E$, one sees that these quaternionic modular forms are nonzero.  Thus, by Dalal's dimension formula \cite{dalal}, these quaternionic modular forms must be $F_9$ and $F_{11}$.  It is proved in \cite{pollackETF} that, for each $m$, the span of the $\Theta_I(X,Y;m)$ and the $\Theta_{E}(X_E,Y_E;m)$ has a basis consisting of cusp forms with integral Fourier coefficients.  Consequently, $F_{9}$ and $F_{11}$ can be normalized to have Fourier coefficients in $\Z$.  It is always true that all Fourier coefficients of $\Theta_E(X_E, Y_E;m)$ corresponding to cubic rings $S = \Z \times B$ are $0$; this follows from \cite[Proposition 5.5]{elkiesGrossIMRN}. This concludes the proof.
\end{proof}

Fix an integral binary cubic form $f(u,v) = au^3+bu^2v+ cuv^2+ dv^3$.  In order for the above Theorems to make sense, one needs to know that the set of rank one $w \in W_J(\Z)$ with $pr_I(w) = f$ is finite; likewise for $pr_E(w)$.  This is true.  To make the SAGE implementation work, we need an explicit form of this finiteness, which we give in the following lemma for the case when the binary cubic form $f$ is monic, i.e., when $a=1$.

\begin{lemma} Suppose $f(u,v)$ is as above, with $a=1$.  Let $\Omega_{f,I} = \{w \in W_J(\Z): \mathrm{rk}(w)=1, pr_I(w) = f\}$ and likewise for $\Omega_{f,E}$.  Then 
\[\Omega_{f,I}=\{(1,T,T^\#,n_J(T)): (T,I^\#) = b, (T^\#,I) = c, n_J(T) =d\}\]
and
\[\Omega_{f,E}=\{(1,T,T^\#,n_J(T)): (T,E^\#) = b, (T^\#,E) = c, n_J(T) =d\}.\]
Moreover, if $(1,T, T^\#, n_J(T)) \in \Omega_{f,I}$, then $(T,T)_I = b^2 - 2c$.  Likewise, if  $(1,T, T^\#, n_J(T)) \in \Omega_{f,E}$, then $(T,T)_E = b^2 - 2c.$
\end{lemma}
Note that the lemma implies the finiteness in an explicit way, because the quadratic forms $(T,T)_I$ and $(T,T)_E$ are positive-definite.  The lemma is an easy consequence of properties of rank one elements (see, for instance, \cite{pollackNotes}) and the definitions.

The SAGE file \texttt{g2\_motives.sage} includes the function \texttt{Dalal\_dim\_k(k)}, that takes Dalal's explicit formula \cite{dalal} for the dimension of the space of level one QMFs on $G_2$ of weight $k \geq 3$ and puts it into the computer.  The smallest weight in which there is a nonzero cusp form, according to Dalal's formula, is in weight $k = 6$.  In this case, the space of weight $6$ level one quaternionic modular forms is one-dimensional, spanned by an element $F_6$.  Computing with the SAGE code, one finds the $F_6$ can be normalized to have the following Fourier coefficients.  (The computation took about one hour on my laptop.). In the table, the ordered $3$-tuple is the $(b,c,d)$ of the monic binary cubic $f(u,v) = u^3 + bu^2 v + cu^2v + dv^3$.

\[
\begin{array}{|c|c|}\hline f(u,v) = u^3 + bu^2 v + cu^2v + dv^3 & a(f) \\ \hline
(0,-3,-1) & 48600\\
(0,-3,0) & 1620\\
(0,-2,-1) & 15\\
(0,-2,0) & 1680\\
(0,-1,0) & -7\\
(1,-3,-3) & -10080\\
(1,-3,-2) & 25575\\
(1,-3,-1) & 28800\\
(1,-3,0) & -1485\\
(1,-2,-2) & -30\\
(1,-2,-1) & 12600\\
(1,-2,0) & -63\\
\hline
\end{array}
\]

\subsection{SAGE implementation: $(J_R,I)$}

To make the formulas used in Theorem \ref{thm:ThetaI} explicit, so that it can put into SAGE, one uses the following easy lemma.
\begin{lemma} Let $T = [u_1,u_2,u_3,v_1,v_2,v_3]$, $X = [x_0,x_1,x_2;x_3,x_4;x_5]$ and $Y=[y_0,y_1,y_2;y_3,y_4,y_5]$.
\begin{enumerate}
\item $(T,T)_I = u_1^2 + u_2^2 + u_3^2 + 2 n(v_1) + 2n(v_2) +2n(v_3).$
\item $(T,I^\#)= u_1 + u_2 + u_3$
\item $(T^\#,I)= u_1 u_2+u_2u_3 + u_3 u_1 - n(v_1)-n(v_2) - n(v_3)$
\item $\det(T) = u_1 u_2 u_3 - u_1 n(v_1) - u_2 n(v_2) - u_3 n(v_3) + (v_1, v_2, v_3)_{\O}$.
\item $(X,T)_I = x_0 u_1 + x_1 u_2 + x_2 u_3 + (v_1,x_3) + (v_2,x_4) + (v_3,x_5)$
\item $(X,T^\#) = A_1 - A_2 + A_3$ where 
\begin{enumerate}
\item $A_1 = x_0 (u_2 u_3 - n(v_1)) + x_1 (u_3 u_1 - n(v_2)) + x_2 (u_1 u_2 - n(v_3))$
\item $ A_2 = u_3 (x_5,v_3) + u_2 (x_4, v_2) + u_1 (x_3,v_1)$
\item $A_3 = (x_3, v_2, v_3) + (x_4, v_3, v_1) + (x_5, v_1, v_2)$.
\end{enumerate}
\end{enumerate}
and one has similar formulas for $(Y,T)_I$ and $(Y,T^\#)$.
\end{lemma}
Now every piece of the computation of $P_{m,I}((1,T,T^\#,n_J(T));X,Y)$ is completely explicit, including how to find all $T$ with $pr_{I}((1,T,T^\#,n_J(T))) = f(u,v)$.  Indeed, to find all such $T$, it suffices to find all $u_1, u_2, u_3 \in \Z$ and $v_1, v_2, v_3 \in R$ with $u_1^2 + u_2^2 + u_3^2 + 2 n(v_1) + 2n(v_2) +2n(v_3) = b^2 -2c$.  Because this quadratic form is visibly decomposable, to find such $u_i$ and $v_j$, one can do it piecewise for the quadratic forms $x \mapsto x^2$ for $x \in \Z$ and $v \mapsto n_{\O}(v)$ for $v \in R$.  This fact leads to our implementation to compute the Fourier coefficients of $\Theta_I$ being faster than its $\Theta_E$ counterpart.

To find various suitable singular pairs $X,Y$ to use as inputs, the code has implemented the exponential of the elements $\Phi'_{u\wedge v}$ of Lemma \ref{lem:F4nilp}.

\subsection{SAGE implementation: $(J_R,E)$}
To implement the formulas in Theorem \ref{thm:ThetaE}, one uses the following straightforward lemma:
\begin{lemma}\label{lem:TEformulas} Suppose $T = [c_1, c_2, c_3; x_1, v_2, v_3]$.  Then $(E^\#, T) = 2(c_1+c_2+c_3) +(\beta^*,x_1 + v_2+v_3)$ and
\begin{align*} (E,T^\#) &= 2(c_1 c_2 + c_2 c_3 + c_3 c_1 - n(x_1) - n(v_2) - n(v_3)) - (\beta, c_1 x_1 + c_2 v_2 + c_3 v_3) \\ &\,\,+ (\beta, v_2, v_3) + (\beta, v_3, x_1) + (\beta, x_1, v_2).\end{align*}
Moreover, the value $(T,X)_E$ can be computed  using the Gram matrix for the quadratic form $(T,T)_E$.
\end{lemma}

\begin{remark} The $x_1$ in the above lemma is not a typo; we have called that octonion $x_1$, instead of $v_1$, so that the lemma is similar to the variables used in the SAGE code (which, unfortunately, are not completely parallel.)\end{remark}

To compute the $f(u,v)$ Fourier coefficient of $\Theta_E(X_E,Y_E;m)$, one first finds all $T \in J_R$ with $(T,T)_E = b^2-2c$.  This uses SAGE's \texttt{short\_vector\_list\_up\_to\_length} function.  Then, having found all such $T = [c_1, c_2, c_3; x_1, x_2, x_3]$, we group them by those that have the same $c_1, c_2, c_3$ and $x_1$.  This allows some of the computation implicit in Lemma \ref{lem:TEformulas} that is identical for multiple $T$ to be done once, instead of repeatedly.  SAGE computes each $P_{m,E}(w;X_E,Y_E)$ and sums up the results.

To find various suitable singular pairs $X_E, Y_E$ one again uses the exponential of the elements $\Phi'_{u\wedge v}$ of Lemma \ref{lem:F4nilp}, together with the element $\delta_E^{-1}$, which is also implemented in SAGE.

\bibliography{nsfANT2020new}
\bibliographystyle{amsalpha}
\end{document}